\documentclass[12pt]{amsart}
\usepackage{amssymb}

\usepackage{palatino}
\input amssym.def

\usepackage{amsmath, amsfonts}
\usepackage{amssymb}
\usepackage{amscd}
\usepackage[mathscr]{eucal}
\usepackage{palatino}
\newfont{\cyrr}{wncyr10}

\newcommand{\R}{{\mathbb R}}
\newcommand{\C}{{\mathbb C}}

\newcommand{\Q}{{\mathbb Q}}
\newcommand{\N}{{\mathbb N}}

\newcommand{\A}{{\mathscr A}}
\renewcommand{\S}{{\mathscr S}}
\renewcommand{\L}{{\mathcal L}}

\newtheorem{thm}{Theorem}[section]
\newtheorem{lem}[thm]{Lemma}
\newtheorem{prop}[thm]{Proposition}
\newtheorem{cor}[thm]{Corollary}
 
\newcommand{\thmref}[1]{theorem~\ref{#1}}

\begin{document}

\title[Special values]{Transcendental sums related to the 
zeros of zeta functions}

\thanks{Research of the second author was supported by 
an NSERC Discovery grant and a Simons Fellowship.}

\author{Sanoli Gun, M. Ram Murty and Purusottam Rath}

\address[Sanoli Gun]  
{Institute of Mathematical Sciences, 
Homi Bhabha National Institute, 
C.I.T Campus, Taramani, 
Chennai  600 113, India.}
\email{sanoli@imsc.res.in}

\address[M.Ram Murty]
{Department of Mathematics,
Queen's University, Jeffrey Hall, 
99 University Avenue, 
Kingston, ON K7L3N6, Canada}
\email{murty@mast.queensu.ca}

\address[Purusottam Rath] 
{Chennai Mathematical Institute,
Plot No H1, SIPCOT IT Park, 
Padur PO, Siruseri 603103, 
Tamilnadu, India} 
\email{rath@cmi.ac.in}



\begin{abstract}
While the distribution of the non-trivial zeros of the Riemann zeta function
constitutes a central theme in Mathematics, nothing is known about 
the algebraic nature of these non-trivial zeros. 
In this article, we study the transcendental nature of 
sums of the form
$$
\sum_{\rho }  R(\rho) x^{\rho},
$$
where the sum is over the non-trivial zeros $\rho$ 
of $\zeta(s)$, $R(x) \in \overline{\Q}(x) $ is 
a rational function
over algebraic numbers and $x >0$ is a 
real algebraic number.
In particular,  we show that the function 
$$
f(x) = \sum_{\rho }  \frac{x^{\rho}}{\rho}
$$ 
has infinitely many zeros in $(1, \infty)$,
at most one of which is algebraic.
The transcendence
tools required for studying $f(x)$ in the range $x<1$ 
seem to be different
from those in the range $x>1$. For $x < 1$, we 
have the following non-vanishing theorem:
If for an integer $d \ge 1$,  $f(\pi \sqrt{d} x)$ 
has a rational zero in $(0,~1/\pi \sqrt{d})$,
then $$
L'(1,\chi_{-d}) \neq 0,
$$
where $\chi_{-d}$ 
is the quadratic character 
associated to the
imaginary quadratic field $K:= \Q(\sqrt{-d})$.
Finally, we consider analogous questions for
elements in the Selberg class. Our proofs rest on results 
from analytic as well as transcendental number theory.
\end{abstract}

\maketitle        
        
\section{Introduction}
\bigskip

For $s\in \C$ with $\Re(s)>1$, the Riemann zeta-function is defined by
$$
\zeta(s)= \sum_{n=1}^{\infty} \frac{1}{n^s}.
$$
It is well known that $\zeta(s)$ has an analytic continuation 
to the entire complex plane except at $s=1$, where it
has a simple pole with residue $1$. The functional equation for
$\zeta(s)$ is determined by the equation 
$$
\xi(s) = \xi(1-s),
$$ 
where $\xi(s)$ is an entire function defined as 
$$
\xi(s) := \frac{s(s-1)}{2}~\pi^{-\frac{s}{2}} ~\Gamma(\frac{s}{2})~\zeta(s).
$$ 
The values taken by the  Riemann zeta function  at 
positive integers and as well as its location of zeros  have
been studied extensively since the time
of Euler and Riemann.  However  the nature of these special values 
continues to elude us though there has been some success following
the works of Apery, Beukers, Rivoal, Zudilin among others.

On the other hand, the  nature of the non-trivial zeros of the Riemann zeta function
is as mysterious as the special values. Let 
$\mathbb{X}$ denote the set of its non-trivial zeros. 
Nothing is known about this set vis-a-vis transcendence. For instance,
consider the field $F = \Q(\mathbb{X})$. Then one can ask the 
following question:

\smallskip
{\it Is the transcendence degree of $F$ over $\Q$ at least one?}

\smallskip
\noindent

Also, consider the $\Q$-vector space $V$ (inside $\C$)  generated
 by the imaginary
parts of elements in $\mathbb{X}$.
Again, we can ask the seemingly easier question:
\smallskip

{\it Is  the dimension of $V$ over $\Q$  at least two?}
\smallskip

In fact, there is a folklore conjecture that the 
imaginary parts of the non-trivial zeros in the upper 
half plane are linearly 
independent over $\Q$ (see for example
Theorem A of \cite{ingham2}, see also \cite{RS}).

\smallskip
\noindent

One believes that the answers to both these questions should 
be affirmative, but it is not clear if the answers lie within the 
reach of the existing transcendence tools  or we need to 
discover new tools. One of the obstacles
to answer such basic questions is that the Riemann zeta function
does not satisfy any differential equation with algebraic parameters. 
More precisely, a classical result of Voronin \cite{SV} 
asserts that $\zeta(s)$ does not satisfy any equation of the form
$$
\sum_{j=0}^{n} s^j F_j\left ( \zeta(s), \cdots, \zeta^{(n-1)}(s) \right )
= 0
$$
for all $s$ lying on a line $\Re(s) = \sigma$ with
$\sigma \in (1/2, 1)$. Here
$F_j$'s for $j= 0, \cdots, n$ are continuous functions on $\C^n$, 
not all identically zero.
This functional independence of Riemann zeta function renders 
effete the applicability of the 
known general transcendental tools to the question of  the nature 
of non-trivial zeta zeros.

The goal of this note is to study the nature of certain general sums 
related to the zeros of the $\zeta$-function. More generally, 
we also consider
sums related to the zeros of  functions in the Selberg class.

Let us now introduce the type of sums we are interested in. Throughout 
the paper $\overline{\Q}$ will denote
 the field of algebraic numbers in $\C$.
Let $A, B \in \overline{\Q}[t]$ be polynomials.  
We study the transcendental nature of sums of the form
\begin{equation}\label{sum}
\sum_{\rho } \frac{A(\rho)}{B(\rho)} x^{\rho},
\end{equation}
where the sum is over the non-trivial zeros $\rho$ 
of $\zeta(s)$. Here $x> 0$ is a real number.
In this paper, we study the situation where $B(t)$ has only
simple zeros.

As shall be evident, we need to consider
the cases $x > 1$, $x=1$ and $0<x<1$ separately.
It appears that the transcendence input in the study
of the case $x>1$ is different from that of $x<1$.
The case $x=1$ is perhaps more mysterious.
For instance, if the rational function $R(x)=A(x)/B(x)$ satisfies
the functional equation $R(x)=-R(1-x)$, then 
$$\sum_{\rho} \frac{A(\rho)}{B(\rho)} = \frac{1}{2}\sum_\rho
\left( \frac{A(\rho)}{B(\rho)} + \frac{A(1-\rho)}{B(1-\rho)}\right) = 0, $$
since the functional equation for $\zeta(s)$ implies that 
 $\rho$ is a zero  if and only if $1-\rho$ is a zero.

The study of sums involving zeros of the Riemann zeta function
can have deep arithmetic significance. For instance,  
in 1997, Xian-Jin Li \cite{li} obtained a simple criterion (now
known as Li's criterion) linking positivity
of certain sums to the Riemann hypothesis. More precisely, let
$$
\lambda_n := \sum_{\rho} \left( 1 - \left(1 - \frac{1}{\rho}\right)^n \right).
$$
Then the Riemann hypothesis is true if and only if $\lambda_n \geq 0$ for
every natural number $n$ (see also work of Brown \cite{FB}).  
This result has led to a flurry of activity
and a plethora of interesting results have emerged from this.
For example, Bombieri and Lagarias \cite{bomb} derived the following
elegant arithmetic identity.  Define the \textit{Stieltjes constants}
$\gamma_n$ by
$$\zeta(s) = \frac{1}{s-1} + \sum_{n=0}^\infty \frac{(-1)^n}{n!}\gamma_n (s-1)^n.$$

It is not difficult to show that the $\gamma_n$'s are given by the limits
$$
\gamma_n = \lim_{m\to \infty} \left(\sum_{k=1}^m \frac{\log^n k}{k}
- \frac{\log^{n+1} m}{n+1}\right), $$
and these can be viewed as generalizations of the more familiar Euler constant
$\gamma_0= \gamma$.  This allows us to define the related constants $\eta_n$ via
$$-\frac{\zeta'(s)}{\zeta(s)} = \frac{1}{s-1} + \sum_{n=0}^\infty
\eta_n (s-1)^n.  $$
By long division, it is now clear that the $\eta_n$'s can be expressed as
polynomials in the $\gamma_n$'s with rational coefficients.
For instance, 
$$ \eta_0 = -\gamma_0, \quad \eta_1 = -\gamma_1 + \frac{1}{2}\gamma_0^2,
\quad \eta_3 = -\gamma_2 + \gamma_0\gamma_1 - \frac{1}{3}\gamma_0^3. $$
Then, it is shown in \cite{bomb} that
$$\lambda_n = - \sum_{j=1}^n 
\left( {n \atop j} \right) \eta_j + 1 - (\log 4\pi + \gamma)\frac{n}{2}
-\sum_{j=2}^n (-1)^{j-1} \left( {n \atop j} \right) (1-2^{-j})\zeta(j). $$
For $n=1$, this reduces to 
\begin{equation}\label{dagger}
\sum_{\rho} \frac{1}{\rho} = 1 + \frac{\gamma}{2} - \frac{\log 4\pi}{2}
= .0230957... 
\end{equation}
In this context, one has the following curious equivalence: 
$$
\text{ Riemann hypothesis }  
\iff
\sum_{\rho} \frac{1}{|\rho|^2} = 2 + \gamma - \log 4\pi 
$$
This is easy to show.  Indeed, 
\begin{equation}\label{starstar}
2\sum_{\rho} \frac{1}{\rho} = \sum_{\rho} \left( {1 \over \rho} + {1 \over
\overline{\rho}}\right) = \sum_{\rho} {2\Re(\rho) \over |\rho|^2} 
\end{equation}
from which we immediately see the result if the Riemann hypothesis
is true.  For the converse, suppose $\rho= x+iy$, with $x,y \in \R$
is a zero such that $x>1/2$.  Then, $(1-x)^2 + y^2 < x^2 + y^2$
so that 
$$(x - 1/2) \left( {1 \over (1-x)^2 + y^2} - {1 \over x^2 + y^2} \right)
>0. $$
In other words, for $x > 1/2$, 
\begin{equation} \label{star}
{1\over 2} \left( {1 \over (1-x)^2 + y^2} + {1 \over x^2 + y^2}\right)
> {x \over x^2 + y^2} + {1-x\over (1-x)^2 + y^2}.
\end{equation}
Writing our sum as
$$\sum_{\rho} {1 \over |\rho|^2} = \sum_{\rho, \Re(\rho)=1/2} {1 \over
|\rho|^2} + \sum_{\rho, \Re(\rho)\neq 1/2} {1 \over
|\rho|^2}$$
and pairing the zero $\rho$ with $1-\rho$ in the second sum, we deduce 
from (\ref{star}) that
$$\sum_{\Re(\rho)\neq 1/2} {1 \over |\rho|^2} > 2\sum_{\Re(\rho)\neq 1/2}
{\Re(\rho)\over |\rho|^2} . $$
Since
$$\sum_{\rho, \Re(\rho)=1/2} {1 \over
|\rho|^2} = 2\sum_{\Re(\rho)\neq 1/2}
{\Re(\rho)\over |\rho|^2} $$
we see that if the Riemann hypothesis is false, then using (\ref{starstar})
and (\ref{dagger}), 
$$\sum_{\rho} {1 \over |\rho|^2} > \sum_{\rho} {2\Re(\rho) \over |\rho|^2}
= 2 + \gamma - \log 4\pi $$
contrary to our hypothesis.  This idea can be easily 
generalized to the Selberg class
(see \cite{droll} for details).

Related to this, we study the possible transcendental nature
of sums of the form 
$$
\sum_{\nu >0} \frac{\cos(\nu \log x)}{\frac{1}{4} + \nu^2}
$$
for any algebraic $x>1$, subject to Riemann hypothesis.
Here, the sum is over the positive imaginary parts of the 
non-trivial zeros of the zeta function.  Similar investigations 
are also carried out without the assumption of the Riemann hypothesis.

The expression for $\lambda_n$ has been studied by several authors
from various angles.  Coffey \cite{coffey} writes
$$\lambda_n = 1 - \frac{n}{2}(\gamma + \log 4\pi) +S_1(n) + S_2(n), $$
where
$$
S_1(n) = \sum_{j=2}^n \left( {n \atop j}\right) (-1)^j 
\left( 1 - \frac{1}{2^{j}}\right)\zeta(j)$$
and $$S_2(n) = \sum_{j=1}^n \left( {n \atop j}\right) \eta_{j-1}. $$
Coffey showed that for $n \geq 2$,
$$ \frac{1}{2}(n(\log n + \gamma -1) +1) \leq S_1(n) \leq \frac{1}{2}
(n(\log n + \gamma+ 1) -1). $$
In particular, $S_1(n)$ is non-negative for every $n \geq 2$.  
This theorem reduces the study of $\lambda_n$ to the study of
$S_2(n)$ and sums involving the Stieltjes constants.

Bombieri and Lagarias \cite{bomb} show that the condition of positivity
can be considerably weakened to deduce the Riemann hypothesis.  In fact,
they show that if for any $\epsilon >0$, there is a constant
$c(\epsilon) >0$ such that
$$\lambda_n \geq -c(\epsilon) e^{\epsilon n} $$
for every $n \geq 1$, then the Riemann hypothesis follows.
Estimates for the Stieltjes constants have been studied
by several authors (see for example, \cite{coffey2}),
but these estimates give super-exponential estimates for
the sums in question.

Though the prototypical zeta function is the Riemann zeta function,
it is useful and interesting to consider the more general setting
of the Selberg class $\S$ which we carry out in this paper.  

Before we proceed further, let us fix some notations.
Throughout the paper, we denote
by  $\rho_F$ or sometimes by $\rho$ (if the context is clear)
the non-trivial zeros of an element $F(s)$ in the Selberg class.
In this context, we examine sums of the form (\ref{sum}) when
$\rho$ runs through zeros of a fixed element of the Selberg class.
More details of this theory can be found in section 6 below.

\section{Some Transcendental Prerequisites}

First we recall the following theorem
due to Alan Baker which will play a key role in our investigation.

\begin{thm} \label{baker} {\rm (Baker)}  If $\alpha_1, ..., \alpha_m$
are non-zero algebraic numbers such that 
$\log \alpha_1,  ..., \log \alpha_m$ are  linearly independent over $\Q$,
then $$1, \log \alpha_1, ..., \log \alpha_m$$ are linearly independent
over $\overline{\Q}$.  
\end{thm}

Let $\L$ denote the $\overline{\Q}$-vector space 
generated by the logarithms of non-zero algebraic
numbers. We 
refer this as the space of Baker periods. Baker's theorem 
asserts that
every non-zero Baker period is transcendental.

We shall call the elements in the $\overline{\Q}$-vector space generated by 
the logarithms of non-zero algebraic numbers and $1$ as 
{\it extended Baker periods.}

We now recall the following far reaching conjecture in 
transcendence theory due to S. Schanuel.

\smallskip

{\it {Schanuel's Conjecture:} Suppose that 
$\alpha_1, \cdots, \alpha_n$ 
are complex numbers which are linearly 
independent over $\Q$.
Then the transcendence degree of the field
$$
\Q(\alpha_1, \cdots, \alpha_n,e^{\alpha_1}, 
\cdots, e^{\alpha_n})
$$
over $\Q$ is at least $n$}.

We shall need the following consequence 
of the above 
conjecture which is not difficult to deduce.
This was done in an earlier work of ours
(see \cite{GMR}
for details).

\smallskip
\noindent
\begin{prop}\label{sch}
Assume that Schanuel's conjecture is true. If 
$\alpha_1, \cdots, \alpha_n$ are 
non-zero algebraic numbers such that 
$\log\alpha_1, \cdots, \log\alpha_n$ are linearly independent 
over $\Q$, then 
$$
\log\alpha_1, \cdots, \log\alpha_n, \log\pi
$$ 
are algebraically independent. 
In particular, $\log{\pi}$ is a transcendental number 
which is not an extended Baker period.
\end{prop}

Finally, we shall need the following theorem of Nesterenko
(\cite{nes1}, see also \cite{nes2}).
\begin{thm} \label{nester}
Let $\wp(z)$ be a Weierstrass $\wp$-function with
\index{Weierstrass $\wp$-function}
algebraic invariants $g_2, g_3$ and with complex multiplication
by an order of an imaginary quadratic field $K$.  Let
$\omega$ be a non-zero period and  $\eta$ the 
corresponding quasi-period.
Then for any $\tau \in K$ with $\Im(\tau) \neq 0$,  
each of these sets
$$\{ \pi, \omega, e^{2\pi i \tau} \} \phantom{m} 
{\rm and} \phantom{m}   \{ \omega, \eta,
e^{2\pi i \tau} \} $$
is algebraically independent over $\Q$.  
\end{thm}

\section{The case of the Riemann zeta function}

We begin by considering the following function
$$
f : (0, \infty) \to \C
$$
given by 
$$
f(x) =~~ \sum_{\rho} \frac{x^{\rho}}{\rho} ~~:=~~ 
\lim_{T \to \infty}~~ \sum_{|t| < T} \frac{x^{\rho}}{\rho},  
$$
where $\rho = \sigma + i t $ runs over the non-trivial 
zeros of the Riemann zeta function
in the critical strip $0 < \Re(s) < 1$.
Recall that 
$$
f(1) = \frac{1}{2} \gamma + 1 - \frac{1}{2} \log 4\pi,
$$ 
where $\gamma$
is the Euler's constant. It is not known whether
$f(1)$ is an irrational number.

We are interested in studying the values taken 
by the function $f$ at algebraic points. 
We first have the following theorem:

\begin{thm}
The set  $X$ given by 
$$ 
X := \{ f(x) : x \in (1,\infty) 
\cap \overline{\Q}  ~~\}
$$
has at most one algebraic element.
\end{thm} 
\begin{proof}
For $x >1$, consider the function $\psi_0(x)$ given by
$$
\psi_0(x) 
= 
\begin{cases}
\sum_{n \leq x} \Lambda(n) & \text{if $x$ is not a prime power;}\\
\sum_{n \leq x} \Lambda(n) - \frac{1}{2}\Lambda(x) & \text{otherwise.}
\end{cases}
$$
Observe that $\psi_0(x)$ is a Baker period.  
Here 
$$
\Lambda(n) =
\begin{cases}
\log p & \text{if $n$ is a power of a prime number $p$;}\\
0 & \text{otherwise.}
\end{cases}
$$ 
is the classical von Mangoldt function.

When $x >1$, one has
the following explicit formula of 
von Mangoldt (see page 77 of \cite{AI}, 
for instance)
$$
f(x) = x - \psi_0(x) - \log 2\pi - 
\frac{1}{2}\log \left( 1 - \frac{1}{x^2}\right).
$$

For  $x>1$, consider the function
$$
g(x) = f(x) - x + \log 2\pi = - \psi_0(x) - 
\frac{1}{2}\log \left( 1 - \frac{1}{x^2}\right).
$$
Note that $g(x)$ is a strictly decreasing 
function in $(1,\infty)$.
Now suppose that $f(x)$ is algebraic at 
two distinct algebraic points,
say $\alpha$ and $\beta$. Then 
$$
g(\alpha) - g(\beta) = f(\alpha) - \alpha - f(\beta) + \beta
$$ 
is a non-zero algebraic number. But this 
is also a Baker period,
a contradiction.
\end{proof}

\begin{thm}
The function $f$ has infinitely many zeros in $(1, \infty)$ 
of which at most one is algebraic.
\end{thm}
\begin{proof}
Note that for any algebraic $x>1$,
$f(x) - x + \log 2\pi$ is a Baker period. Thus 
by Baker's theorem, $f$ is injective on the
set of algebraic numbers greater than 1. Hence $f$ can have 
at most one algebraic zero. 

We now show that $f$ has infinitely many zeros in $(1, \infty)$.
Let us write $f$ as 
$$
f(x) = h(x) -  \psi_0(x)
$$
where 
$h(x) = x   - \frac{1}{2}\log \left( 1 - \frac{1}{x^2}\right) - \log 2\pi. $ 

Since $\psi_0(n) - n = \Omega_{\pm}(n^{\frac{1}{2}})$ 
(see \cite{AI}, page 91), the sequence
$\left\{f(n)\right\}_{n \in \N}$ changes sign infinitely often. 
In particular, there exists infinitely many $N \in \N$ for which
$f(N) < 0$ while $f(N+1)~>~0$.
Note that  $f$ is continuous in $(1, \infty)$ except at prime 
powers where it is right continuous.
Also $h(x)$ is a strictly increasing continuous function in $(1, \infty)$.
Let $N_0 >1$ be a natural number such that $f(N_0) < 0$ while
$f(N_0+1)~>~0$. If $f(x) \geq 0$ for some $x \in (N_0, N_0+1)$, 
we have a zero
of $f$ in $(N_0, N_0+1)$. Assume otherwise. Since 
$\psi_0(x)$ is non-negative and constant in $[N_0,N_0+1)$, 
$N_0+1$ cannot be a prime power. Thus the function $f$ is continuous
in $[N_0, N_0+1]$ and hence must have a zero in this interval.
\end{proof}

We also have the following conditional result.

\begin{thm}
Assume Schanuel's conjecture. Then $X$ has no algebraic element.
\end{thm} 
\begin{proof}
Suppose that both $f(x)$ and $x$ are algebraic. Then  $\log \pi$ lies
in the $\overline{\Q}$-vector space generated by logarithms of 
non-zero algebraic numbers
and $1$. But by Proposition \ref{sch}, this is not possible if 
we assume Schanuel's conjecture.
\end{proof}

We now consider the case for $0 <x <1$.
When $0< x< 1$, one has the following expression 
as indicated by (Ingham \cite{AI}, page 81):
$$
{\sum_{n \leq 1/x}}'~~ \frac{\Lambda(n)}{n} 
= 
-\log x - \gamma ~+~ \sum_{\rho}\frac{x^\rho}{\rho}
~+~ \frac{1}{2}\log \frac{1+x}{1-x} -x, 
$$
where $\gamma$ denotes the Euler's constant.
The dash in the sum means that
 there is a correction factor of $1/2$ in the last term
of the sum involving the von Mangoldt function
when $x$ is the reciprocal of some prime power.

The above expression can be deduced  by considering 
the following integral
$$
\frac{1}{2\pi i}\int_{3-i \infty}^{3+ i \infty} \frac{x^{1-s}}{1-s} 
~\frac{\zeta'}{\zeta}(s) ~ds.
$$  
By Perron's formula, this integral is equal to the left hand 
side of the above expression.
As with the explicit formula for $x>1$, completing this integral 
into a rectangular contour,
we will have contributions exactly from the residues of the 
poles of $\frac{x^{1-s}}{1-s} ~\frac{\zeta'}{\zeta}(s)$ in 
the complex plane. The double pole at $s=1$ contributes
the  factor $-\log x - \gamma$. The poles from 
non-trivial zeros contribute the factor
$$
\sum_{\rho}\frac{x^{1-\rho}}{1-\rho} 
= \sum_{\rho}\frac{x^\rho}{\rho}.
$$ 
Finally, the trivial zeros of the zeta function contribute 
$$ 
\sum_{n=1}^{\infty} \frac{x^{2n+1}}{2n+1}
=
\frac{1}{2}\log \frac{1+x}{1-x} -x.
$$

Arguing as  earlier, we can now deduce the following result:
\begin{thm}
The set  $Y$ given by 
$$ 
Y := \{ f(x) - x : x \in (0,1) \cap \overline{\Q}  ~~\}
$$
has at most one algebraic element. In particular,
$f$ has at most one algebraic zero in $(0,1)$.\end{thm} 
\begin{proof}
For $0 < x< 1$,
$$ 
f(x) - x - \gamma  = {\sum_{n\leq 1/x}}'~~ \frac{\Lambda(n)}{n} 
+\log x  - \frac{1}{2}\log \frac{1+x}{1-x}  
$$
and hence a Baker period if $x$ is algebraic.
If there are two algebraic values in the set, we argue as we did in our
earlier theorem.
\end{proof}

Also, we immediately observe the following:
\begin{cor}
If $f(x)$ is algebraic for some algebraic $x$ in $(0,1)$,
then $\gamma$ is transcendental.
\end{cor}

It seems that the  existence of the (presumably) fictitious 
algebraic element in the
above theorem cannot be ruled out under Schanuel's conjecture. 
Thus the transcendence
tools required for studying $f(x)$ in the range $x<1$ 
seem to be different
from those in the range $x>1$.
We however have the following curious theorem.

\begin{thm}\label{main}  
For an integer $d \ge 1$, suppose that $f(\pi \sqrt{d} x)$ 
has a rational zero in $(0,~1/\pi \sqrt{d})$. Then 
for the quadratic character $\chi_{-d}$ 
associated to the
imaginary quadratic field $K:= \Q(\sqrt{-d})$,
one has
$$
L'(1,\chi_{-d}) \neq 0.
$$
\end{thm}
\begin{proof}
As discussed above for $0 < x< 1$,
$$
 f(x) = {\sum_{n\leq 1/x}}^{'}~~ \frac{\Lambda(n)}{n} 
+ \log x + \gamma - \frac{1}{2}\log \frac{1+x}{1-x} + x .
$$

Suppose $f(\pi \sqrt{d} r) = 0$ for some rational $r \in (0,1/\pi \sqrt{d})$. 
Write $x = \pi r \sqrt{d}$.
Then we have
$$
\gamma = -{\sum_{n\leq 1/x}}^{'}~~ \frac{\Lambda(n)}{n} 
-\log x 
+ \frac{1}{2}\log \frac{1+x}{1-x} -x.
$$
Therefore,
\begin{equation}\label{a}
e^{2\gamma} =  \alpha e^{-2\pi r \sqrt{d}} 
\end{equation}
where $\alpha \in \overline{\Q}(\pi)$.
 
Now let $\chi_{-d}$ be the quadratic character associated 
to the imaginary 
quadratic field $K = \Q(\sqrt{-d})$ such that $L'(1,\chi_{-d}) = 0$.
 
It is known (see \cite{murty-murty}, page 848) that 
(essentially by the Chowla-Selberg formula)
 $$
 \exp\left( {L'(1,\chi_{-d}) \over L(1,\chi_{-d})} - \gamma\right)
= (2D/A^2) \prod_{a=1}^D \Gamma(a/D)^{-\chi_{-d}(a)w/2h} 
 $$
where $A=\sqrt{D/\pi}$, $D$ is the absolute discriminant
of $K$, $h$ and $w$ are the class number and order
of unit group of $K$ respectively.

As observed by Gross \cite{gross1}, the number
$$ 
\prod_{a=1}^D \Gamma(a/D)^{\chi_{-d}(a)} 
$$
is, up to an algebraic factor, equal to a product of a 
power of $\pi$ and a power
of a non-zero period $\omega$ of the 
CM elliptic curve attached to
the full ring of integers of $K$.
 
Since $L'(1,\chi_{-d}) = 0$, we see from above that 
$e^\gamma \in \overline{\Q}(\pi , \omega)$. 
On the other hand, from \eqref{a}, we have 
$
e^{2\gamma} =  \alpha e^{-2\pi r \sqrt{d}} 
$
with $\alpha \in \overline{\Q}(\pi)$. This contradicts Nesterenko's 
result (Theorem \ref{nester}). 
\end{proof}

As evident, while $\log \pi$ is the mysterious number that shows 
up in the evaluation of $f(x)$ for $x>1$, it is
$\gamma$ that enters the picture for $x<1$.
We would like to obtain transcendence results
involving both $\log \pi$ and $\gamma$.
For $x >1$, replacing $x$ by $1/x$ in Ingham's formula, we obtain 
$$
L(x):={\sum_{n\leq x}}' \frac{\Lambda(n)}{n} = \log x  - \gamma
+ \sum_{\rho} \frac{x^{-\rho}}{\rho} 
- \frac{1}{2}\log \frac{x+1}{x-1} -\frac{1}{x}. 
$$
Recall that for such an $x$, we have
$$
\psi_0(x) = x - \sum_{\rho} \frac{x^{\rho}}{\rho} 
-\log 2\pi -\frac{1}{2}\log \left( 1 - \frac{1}{x^2}\right).
$$
Assuming the Riemann hypothesis so that a typical 
zero $\rho$ is of the form $1/2 + i\nu$,  
and pairing the zeros $1/2 + i\nu$ with $1/2 - i\nu$,
we obtain the following expression for $x >1$:
\begin{eqnarray*}
\sum_{\nu >0} \frac{2\cos(\nu \log x)}{\frac{1}{4} + \nu^2}
&=& \sum_{\nu} \frac{x^{i\nu} + x^{-i\nu}}{\frac{1}{2} + i\nu}
= \frac{x-\psi_0(x)}{\sqrt{x}} - \frac{\log 2\pi}{\sqrt{x}}
-\frac{1}{2\sqrt{x}}\log \left( 1- \frac{1}{x^2}\right)\\
&&
+\sqrt{x}(L(x)-\log x) +\gamma\sqrt{x} -\frac{\sqrt{x}}{2}\log
\frac{x+1}{x-1}+ \frac{1}{\sqrt{x}}  \\
\text {and hence } \phantom{mmmmm} && \\
\sum_{\nu > 0} \frac{2\cos(\nu \log x)}{\frac{1}{4} + \nu^2} 
&+&  \frac{\log 2\pi}{\sqrt{x}} - \gamma\sqrt{x} 
= \frac{x-\psi_0(x)}{\sqrt{x}} 
-\frac{1}{2\sqrt{x}}\log \left( 1- \frac{1}{x^2}\right) \\
&&
+ \sqrt{x}(L(x)-\log x)  - \frac{\sqrt{x}}{2}\log
\frac{x+1}{x-1} + \frac{1}{\sqrt{x}}.
\end{eqnarray*}
We now have the following: 
\begin{thm}  
Assume the Riemann hypothesis.  For any algebraic $x>1$, 
$$
\sum_{\nu >0} \frac{2\cos(\nu \log x)}{\frac{1}{4} + \nu^2}
+\frac{\log 2\pi}{\sqrt{x}} - \gamma \sqrt{x} $$
is an extended Baker period. 
If Schanuel's conjecture is true, then the following set 
$$
\left\{ \sum_{\nu >0} \frac{\cos(\nu \log x)}{\frac{1}{4} + \nu^2} 
: x >1, x \in \overline{\Q} \right\}
$$
has at most one algebraic number.
\end{thm}

We now derive a related result without assuming the 
Riemann hypothesis.  
To this end, we observe that we can write
$$
\sum_{\rho} \frac{x^{-\rho}}{\rho}=  \sum_{\rho}{x^{-(1-\rho)}
\over 1-\rho} =  \sum_{\rho} {x^{\rho-1} \over 1-\rho}, 
$$
by virtue of the functional equation. Thus
$$
L(x) = \log x - \gamma - \sum_{\rho} {x^{\rho -1} \over \rho}
+ \left( \sum_{\rho}{x^{\rho-1} \over \rho} + 
\sum_{\rho} {x^{\rho-1} \over 1-\rho} \right) 
+ {1\over 2} \log {x+1 \over x-1} - \frac{1}{x}.
$$

The sum in brackets can be written as 
$$
\sum_{\rho} {x^{\rho-1} \over \rho(1-\rho)} 
$$
which is an absolutely convergent series and is thus equal to
$$
\sum_{\rho} {x^{\rho-1} \over \rho(1-\rho)} =
L(x) -\log x + \gamma + \sum_{\rho} {x^{\rho -1} \over \rho}
- {1\over 2} \log {x+1 \over x-1} + \frac{1}{x}.
$$
The sum
$$
\sum_{\rho} {x^{\rho -1} \over \rho} 
$$
is equal to
$$
{x -\psi_0(x) \over x} - {\log 2\pi \over x}
 - {1 \over 2x} \log \left( 1 - {1 \over x^2}\right)
$$ 
and hence
\begin{eqnarray*}
\sum_{\rho} {x^\rho \over \rho(1- \rho)} -\gamma x + \log 2\pi 
&=&
1 + x(L(x)-\log x ) + x - \psi_0(x)  \\
\phantom{jajamsmsmsmsmsm} 
&& - {x \over 2} \log {x+1 \over x-1}
-{1 \over 2} \log \left( 1 - {1 \over x^2}\right).
\end{eqnarray*}
As the right hand side is an extended Baker period for algebraic
$x$, 
this proves:

\begin{thm}  
For $x >1$, 
$$
S(x): =\sum_{\rho} {x^\rho \over \rho(1- \rho)}
-\gamma x + \log 2\pi
$$
is an extended Baker period.
In particular, assuming Schanuel's conjecture, the set 
$$
\left\{\sum_{\rho} {x^\rho \over \rho(1 - \rho)} : x >1, 
x \in \overline{\Q} \right\}$$
contains at most one algebraic number.
\end{thm}

We remark that an expression similar to ours in the above
theorem was
also obtained by Ramar\'e \cite{ramare} (however, sign in the sum
over the zeros in his Lemma 2.2 should be negative).

\section{Sums of general type involving the Riemann zeta function}

We now consider more general sums of the form
$$
f(x) := \sum_{\rho} \frac{A(\rho)}{B(\rho)} x^{\rho},
$$
where $A(t) \in \overline{\Q}[t]$ while $B(t) \in \Q[t]$ 
be polynomials and $x \in (0, \infty)$. We
assume that $B(t)$ has simple rational
roots with degree greater than that of $A(t)$.
As before, the sum is defined following the convention 
of section 3. 

We shall need the following elementary lemma whose proof 
we omit (see p. 137 of \cite{FN}).  
\begin{lem}\label{nest}
 Let for $|z| < 1$, let
$$
f_u(z) = \sum_{n=1}^\infty \frac{z^n}{n+u}. 
$$
If $u=p/q$ is a rational number, then
$$
f_u(z) = -z^{p/q} 
\sum_{m=0}^{q-1} \zeta_q^{-pm} \log (1 - \zeta_q^m z^{1/q}), 
$$
where $\zeta_q = e^{2\pi i/q}$.
\end{lem}

We first consider the case when $x>1$. For this,
we shall further assume that $B(t)$ has simple rational
 roots lying in $\Q \setminus \{ 1, -2,-4, -6 \cdots\}$.
 We have the following theorem.
 \begin{thm}
 Let $ A(t)$ and $B(t)$ be as described above and let  
 $\alpha_{1}, \cdots, \alpha_{d}$ be the  roots of $B(t)$.
For an algebraic number $x > 1$, 
$$
g(x):= \sum_{\rho} \frac{A(\rho)}{B(\rho)}x^\rho
 + \sum_{i=1}^d \frac{A(\alpha_i)}{B'(\alpha_i)} 
 \frac{{\zeta}'}{\zeta}
 (\alpha_i) x^{\alpha_i}
$$
is an extended Baker period. Further
for $\lambda_{i}:= \frac{A(\alpha_i)}{B'(\alpha_i)}$,
\begin{itemize}
\item
 if $\sum_{i=1}^{d} 
\frac{\lambda_{i}}{1-\alpha_{i}} \neq 0$, then
$g(x)$ has at most one algebraic zero 
in $(1,\infty)$;
\item
if $\sum_{i=1}^{d} 
\frac{\lambda_{i}}{1-\alpha_{i}} = 0$ 
and $g(x) \neq 0$ for some algebraic
$x> 1$, then at least one of the two
numbers
$$
\sum_{\rho} \frac{A(\rho)}{B(\rho)}x^\rho,
\phantom{m}
\sum_{i=1}^d \lambda_i
 \frac{{\zeta}'}{\zeta}
 (\alpha_i) x^{\alpha_i}
 $$
 is transcendental. 
 \end{itemize}
\end{thm}
\begin{proof}
For any $x> 1, \alpha \in \R \setminus \{ 1, -2,-4, -6 \cdots\}$, we 
have the following expression  
 $$
 \psi_0(x, \alpha)
 =  \frac{x}{1-\alpha} - x^{\alpha} \frac{{\zeta}'}{\zeta}
 (\alpha) - \sum_{\rho}\frac{x^\rho}{\rho - \alpha} 
 +\sum_{n=1}^{\infty} \frac{x^{-2n}}{2n+ \alpha},
$$
where
$$
\psi_0(x, \alpha) := 
\begin{cases}
 x^{\alpha} {\sum_{n \leq x}}~~ \frac{\Lambda(n)}{n^{\alpha}} 
 & \text{if $x$ is not a prime power;}\\ \\
x^{\alpha} {\sum_{n < x}}~~ \frac{\Lambda(n)}{n^{\alpha}}  + 
\frac{1}{2}\Lambda(x) & \text{otherwise.}
\end{cases}
$$
This follows by modifying the explicit formula suitably.
We now re-write this as 
\begin{equation}\label{b}
 \sum_{\rho}\frac{x^\rho}{\rho - \alpha} 
 +  x^{\alpha} \frac{{\zeta}'}{\zeta}
 (\alpha) = \frac{x}{1-\alpha} - \psi_0(x, \alpha)
 +\sum_{n=1}^{\infty} \frac{x^{-2n}}{2n+ \alpha}.
\end{equation}
Note that by Lemma \ref{nest}, when $\alpha$
is a rational number and $x$ is algebraic, 
the right hand side of \eqref{b} is 
an extended Baker period. Now by using 
partial fractions, we can write
$$
\frac{A(t)}{B(t)} = 
\sum_{i=1}^{d} \frac{\lambda_{i}}{t-\alpha_{i}}
$$
with $\lambda_{i}:= \frac{A(\alpha_i)}{B'(\alpha_i)}$. 
Thus the function 
$$
\sum_{\rho} \frac{A(\rho)}{B(\rho)}x^\rho ~+~ 
 \sum_{i=1}^d \frac{A(\alpha_i)}{B'(\alpha_i)} 
 \frac{{\zeta}'}{\zeta} (\alpha_i) x^{\alpha_i}
 $$ 
 is equal to
 $$
 x \sum_{i=1}^{d} \frac{\lambda_{i}}{1-\alpha_{i}} 
 - \sum_{i=1}^{d} \lambda_{i} \psi_0(x , \alpha_i) +
 \sum_{i=1}^{d} \lambda_{i} \sum_{n=1}^{\infty}
  \frac{x^{-2n}}{2n+ \alpha_{i}}.
$$
which is an extended Baker period if $x$ is algebraic.
The second part of the theorem is again a
consequence of Baker's theorem.
\end{proof}

\bigskip
 
Finally, when $x \in (0,1)$ and  $\alpha \in \R 
\setminus \{ 0, 1,3,5,7 \cdots\}$, we have 
$$
T(x, \alpha)  
 =\sum_{\rho}\frac{x^\rho}{\rho - \alpha} + \frac{1}{\alpha}  -
 x^{\alpha} \frac{{\zeta}'}{\zeta}
 (1- \alpha) + \sum_{n=1}^{\infty} \frac{x^{2n+1}}{2n+1-\alpha},
 $$
 where
 $$
 T(x, \alpha) := 
\begin{cases}
x^{\alpha} {\sum_{n \leq 1/x}}~~ \frac{\Lambda(n)}{n^{1-\alpha}} 
& \text{if $1/x$ is not a prime power;}\\ \\
x^{\alpha} {\sum_{n < 1/x}}~~ \frac{\Lambda(n)}{n^{1-\alpha}}  + 
\frac{x}{2}\Lambda(1/x) & \text{otherwise.}
\end{cases}
$$
Hence we have 
$$
\sum_{\rho}\frac{x^\rho}{\rho - \alpha}  -x^{\alpha} \frac{{\zeta}'}{\zeta}
(1- \alpha) = T(x, \alpha) -  
\frac{1}{\alpha} -  \sum_{n=1}^{\infty} \frac{x^{2n+1}}{2n+1-\alpha},
$$
 where the right hand side is an extended Baker period 
 when $x$ is algebraic and $\alpha$ is rational. Thus we have
  the following theorem for $x <1$.
 \begin{thm}
 Let $ A(t)$ and $B(t)$ be as before and let  
 $\alpha_{1}, \cdots, \alpha_{d}$ be the  roots of $B(t)$
 all of which lie in  $\Q \setminus \{ 0, 1,3,5,7 \cdots\}$.
For an algebraic $x \in (0,1)$, the number
$$
h(x) := 
\sum_{\rho} \frac{A(\rho)}{B(\rho)}x^\rho
 - \sum_{i=1}^d \frac{A(\alpha_i)}{B'(\alpha_i)} \frac{{\zeta}'}{\zeta}
 (1-\alpha_i)
 x^{\alpha_i}
$$
is an extended Baker period. Further, the set
$$
\{ h(x) ~|~ x \in (0,1)\cap \overline{\Q} \}
$$
can have at most one algebraic number.
\end{thm}

\begin{proof}
As before, using partial fractions, we can deduce that
$$
h(x) =  \sum_{i=1}^{d} \lambda_{i} T(x , \alpha_i) 
- \sum_{i=1}^{d}  \frac{\lambda_{i}}{\alpha_{i}}
-  \sum_{i=1}^{d} \lambda_{i} \sum_{n=1}^{\infty}
  \frac{x^{2n+1}}{2n+ 1 - \alpha_{i}},
$$
where $\lambda_{i}:= \frac{A(\alpha_i)}{B'(\alpha_i)}$.
When $x$ is algebraic, $h(x)$ is an extended Baker period.
Finally the second part of the theorem follows
by noting that for $x, y$ algebraic, 
$h(x) - h(y)$ is a Baker period.
\end{proof}
 
 \section{The Selberg class}

Selberg \cite{selberg} defined a large class $\mathscr{S}$ of Dirichlet series
admitting analytic continuation and functional equation.
It is likely that this class includes the universe of automorphic $L$-functions, 
though this has not yet been proven.  The class $\mathscr{S}$ is
defined as follows.

\begin{enumerate}
\item
Each $F\in {\S}$ is a Dirichlet series 
$$F(s) = \sum_{n=1}^\infty a_F(n)n^{-s}, $$
absolutely convergent for $\Re(s) >1$.
\item  There exists an integer $m \geq 0$ such that
$(s-1)^m F(s)$ is an entire function of finite order.
Let $m_F$ denote the least value of such $m$.
\item  For each $F \in {\S}$, 
there exist  numbers $Q_F>0$ and $r \geq 0$, 
and numbers $\lambda_j >0$ and $\mu_j$ with $\Re(\mu_j) 
\geq 0$,  such that
$$\xi_F(s) := Q_F^s \prod_{j=1}^r \Gamma(\lambda_js + \mu_j) F(s)$$
satisfies the functional equation
$$\xi_F(s) = w~\overline{\xi_F}(1-s), $$
with $w$ a complex number of absolute value 1.
Here $\overline{\xi_F}(s) = \overline{\xi_F(\overline{s})}$ and an 
empty product equals 1.
\item  The Dirichlet coefficients $a_F(n)$ satisfy $a_F(n) \ll n^\epsilon$
for every $\epsilon >0$;
\item $\log F(s)$ can be written as a Dirichlet series
$$\sum_{n=1}^\infty b_F(n)n^{-s}, $$
where $b_F(n)$ is zero unless $n$ is a prime power and
$b_F(n) \ll n^\theta$ for some $\theta <1/2$.
\end{enumerate}

There are several celebrated conjectures related to this class
and we refer the reader to \cite{selberg} and \cite{murty}
for further details.  Because of the Legendre
duplication formula for the $\Gamma$-function, it is easy to
see that  the functional equation is not unique for an arbitrary
element $F$ in $\S$.  However, the invariants
$$
d_F = 2\sum_{j=1}^r \lambda_j, \quad q_F = (2\pi)^{d_F}Q_F^2\prod_{j=1}^r
\lambda_j^{2\lambda_j}, \quad \theta_F = 2\Im\left( \sum_{j=1}^r (\mu_j - 1/2)
\right), 
$$
are well-defined and called, the {\it degree}, the {\it conductor}
and {\it shift}, respectively.
One conjectures that $d_F$ and $q_F$ are positive integers.
Recently, some impressive work \cite{k-p} has appeared that shows that
$0< d_F < 1$ and $1 < d_F < 2$ are impossible.

We now derive a general formula for an element in the Selberg class.
It is convenient to write
$$
-\frac{F'}{F}(s) = \sum_{n=1}^\infty \Lambda_F(n)n^{-s}, \quad
\Lambda_F(n) = b_F(n)\log n. 
$$
For $x > 1$, let us introduce the notation
$$
\psi_0(x,F, \alpha) := 
\begin{cases}
 x^{\alpha} {\sum_{n \leq x}}~~ \frac{\Lambda_F(n)}{n^{\alpha}} 
 & \text{if $x$ is not a prime power;}\\ \\
x^{\alpha} {\sum_{n < x}}~~ \frac{\Lambda_F(n)}{n^{\alpha}}  + 
\frac{1}{2}\Lambda_F(x) & \text{otherwise.}
\end{cases}
$$

Let $\alpha$ be a complex number not equal
to any of the poles and zeros of $F(s)$.
Again the explicit formula for $F$  (see formula (12) of 
\cite{murty-perelli}, for instance)
yields that
$$
\psi_0(x,F, \alpha)  = -x^{\alpha}\frac{F'}{F}(\alpha)
+ \frac{m_F x}{1-\alpha} - \sum_{\rho} \frac{x^{\rho }}{\rho-\alpha}
+ \sum_{j=1}^r \sum_{n=0}^\infty \frac{x^{-((n+\mu_j)/\lambda_j)}}{\frac{n+\mu_j}{
\lambda_j} + \alpha} - \frac{m_F }{\alpha}, 
$$
where $\rho$ runs over the non-trivial zeros of $F$ 
in the sector $0 \leq \Re(s) \leq 1$.
Recalling the following series introduced earlier
$$
f_u(z) = \sum_{n=1}^\infty \frac{z^n}{n+u}, 
$$
we have:
$$ 
\sum_{n=1}^\infty \frac{x^{-((n+\mu_j)/\lambda_j)}}{\frac{n+\mu_j}{
\lambda_j} + \alpha} = x^{-\mu_j/\lambda_j}\lambda_j f_{\mu_j 
+ \alpha\lambda_j}(x^{-1/\lambda_j}), 
$$
and hence
$$
x^{\alpha}\frac{F'(\alpha)}{F(\alpha)}
+  \sum_{\rho} \frac{x^{\rho }}{\rho-\alpha}
$$
is equal to 
$$
=  \frac{m_F x}{1-\alpha}
-\psi_0(x,F, \alpha)  
+ \sum_{j=1}^r x^{-\mu_j/\lambda_j}\lambda_j f_{\mu_j 
+ \alpha\lambda_j}(x^{-1/\lambda_j}) 
+ \sum_{j=1}^r \frac{x^{-(\mu_j/\lambda_j)}}{\frac{\mu_j}{
\lambda_j} + \alpha}
- \frac{m_F }{\alpha}. 
$$
Again let $A(t)\in \overline{\Q}[t]$ and $B(t)\in {\Q}[t]$ be polynomials
such that $B(t)$ has simple rational roots  not equal
to any of the poles and zeros of $F(s)$ and degree of $B(t)$
is strictly greater than the degree of $A(t)$. Then as before,
$$
\frac{A(t)}{B(t)} = 
\sum_{i=1}^{d} \frac{\beta_{i}}{t-\alpha_{i}}
$$
with $\beta_i:= \frac{A(\alpha_i)}{B'(\alpha_i)}$,
and hence we have 
$$
\sum_{\rho} \frac{A(\rho)}{B(\rho)}x^{\rho} ~+~ 
 \sum_{i=1}^d \frac{A(\alpha_i)}{B'(\alpha_i)} 
 \frac{F'}{F} (\alpha_i) x^{\alpha_i}
$$
is equal to
\begin{eqnarray*}
m_F x\sum_{i=1}^d   \frac{\beta_i }{1-\alpha_i} 
 & -& \sum_{i=1}^d \beta_i \psi_0(x,F, \alpha_i)
+ \sum_{i=1}^d \beta_i \left\{\sum_{j=1}^r x^{-\mu_j/\lambda_j}
\lambda_j f_{\mu_j + \alpha_i\lambda_j}(x^{-1/\lambda_j}) \right\}  \\
 &&
 \phantom{mm}
 + ~~~
 \sum_{i=1}^d \beta_i\left\{ \sum_{j=1}^r 
 \frac{x^{-(\mu_j/\lambda_j)}}{\frac{\mu_j}{
\lambda_j} + \alpha_i}
- \frac{m_F }{\alpha_i}\right\}. 
\end{eqnarray*}

We will say two functions $F,G\in \S$ are of the same {\it Hodge type}
if they admit a functional equation
with the same  $\lambda_j$ and the $\mu_j$.  
In such a situation, we can prove the following.

\begin{thm}\label{hodge}
Consider the set of elements $F\in \S$ with a fixed Hodge type
such that 
$$
\Delta_F := \sum_{\rho} \frac{A(\rho)}{B(\rho)}x^\rho
 + \sum_{i=1}^d \frac{A(\alpha_i)}{B'(\alpha_i)}
 \frac{F'}{F}(\alpha_i)x^{\alpha_i}
$$
is algebraic for algebraic $x >1$. Further assume 
that $\beta_i=\frac{A(\alpha_i)}{B'(\alpha_i)}$
are real with same sign for $1\le i \le d$.  
If there are two elements $F,G$ in this set, then for any 
prime $p$ with $\sqrt{x} < p \le x$, we have
$$
b_F(p) =  b_G(p). 
$$
In particular, if $2\leq x < 3$, then $b_F(2)=b_G(2)$.
\end{thm}

\begin{proof}
For any such elements $F$ and $G$, the above discussion along
with Baker's theorem would necessarily imply that
$$
\sum_{i=1}^d \beta_i \psi_0(x,F, \alpha_i)
= 
\sum_{i=1}^d \beta_i \psi_0(x,G, \alpha_i).
$$
The theorem then follows as logarithms of primes are linearly independent
over $\Q$ and hence over $\overline{\Q}$ by Baker's theorem.
\end{proof}

\section{The arithmetic Selberg class $\A$ }

We now focus our attention on  a subclass $\A$ of $\mathscr{S}$, 
which we call the {\it arithmetic Selberg class}.  
The class $\A$ is defined by the following
axioms:
\begin{enumerate}
\item
Each $F\in {\A}$ is a Dirichlet series 
$$
F(s) = \sum_{n=1}^\infty a_F(n)n^{-s}, 
$$
absolutely convergent for $\Re(s) >1$.
\item  
There exists an integer $m \geq 0$ such that
$(s-1)^m F(s)$ is an entire function of finite order.
As before, $m_F$ is the smallest such $m$.
\item  For each $F \in {\A}$, 
there exist numbers $Q_F$ and $r$, 
and \textit{rational numbers} $\lambda_j >0$ and $\mu_j \geq 0$ such that 
$$
\xi_F(s) := Q_F^s \prod_{j=1}^r \Gamma(\lambda_js + \mu_j) F(s)
$$
satisfies the functional equation
$$
\xi_F(1-s) = w \overline{\xi_F}(1-s), 
$$
with $w$ a complex number of absolute value 1.
Moreover, the conductor $q_F$ is assumed to be a \textit{natural number}.
\item $\log F(s)$ can be written as a Dirichlet series
$$\sum_{n=1}^\infty b_F(n)n^{-s}, $$
with $b_F(n)$ \textit{algebraic} satisfying
$b_F(n)=0$  if $n$ is not a power of $p$, with $p$ prime.
\end{enumerate}

Technically speaking, $\A$ is not a subclass of $\S$ since
the reader will note that the Ramanujan estimate for the coefficients
$a_F(n)$ which appears
in the definition of the Selberg class $\mathscr S$, is not assumed in the 
above definition since it is not essential
for the nature of the theorems we will derive.  
We also do not assume any estimate for $b_F(n)$. 
It is also easy
to see that the algebraicity of $b_F(n)$ implies the
algebraicity of $a_F(n)$.

Most of the zeta functions that arise in number theory
(such as the Artin $L$-functions and zeta functions
attached to algebraic varieties) either belong to $\A$
or are expected to belong to $\A$.  

As earlier, for $F\in \A$ and for $x >1$, 
$$
x^{\alpha}\frac{F'}{F}(\alpha)
~+~  \sum_{\rho} \frac{x^{\rho }}{\rho-\alpha}
$$
is equal to 
$$
=  \frac{m_F x}{1-\alpha} + \sum_{j=1}^r 
\frac{x^{-(\mu_j/\lambda_j)}}{\frac{\mu_j}{\lambda_j} + \alpha}
- \frac{m_F }{\alpha}
-\psi_0(x,F, \alpha)  
+ \sum_{j=1}^r x^{-\mu_j/\lambda_j}\lambda_j 
f_{\mu_j + \alpha\lambda_j}(x^{-1/\lambda_j}) . 
$$

From this formula, we see that for $x$ algebraic,
the right hand side is an extended Baker period
provided $\alpha$, the $\mu_j$'s and the $\lambda_j$'s are all
rational numbers. 
\begin{thm} \label{arith}
Let $F\in \A$.  Let $ A(t), B(t)$ be  polynomials 
as before and $B(t)$ of degree $d$
with simple rational roots $\alpha_1, \cdots, \alpha_d$
not equal to the zeros and poles of $F$. 
For $x > 1$ and algebraic, we have that
$$
g(x):= \sum_{\rho} \frac{A(\rho)}{B(\rho)}x^\rho
 + \sum_{i=1}^d  \beta_i \frac{F'}{F}(\alpha_i) x^{\alpha_i}
$$
where $\beta_{i}:= \frac{A(\alpha_i)}{B'(\alpha_i)}$,
is an extended Baker period. Further suppose that 
$\mu_j =0 ~~~~~\forall j$. Then
\begin{itemize}
\item
 if $m_F\sum_{i=1}^{d}
\frac{\beta_{i}}{1-\alpha_{i}} \neq 0$, then
$g(x)$ has at most one algebraic zero 
in $(1,\infty)$;
\item
if $m_F \sum_{i=1}^{d} 
\frac{\beta_{i}}{1-\alpha_{i}} = 0$, then the
set
$$
\{ g(x) ~|~ x \in (1, \infty)\cap \overline{\Q} \}
$$
has at most one algebraic number.
\end{itemize}
\end{thm}

\begin{proof} 
The proof follows from the preceding discussion 
and appealing to Baker's theorem.
\end{proof}

\section{The case $ 0 < x < 1$ for the Selberg Class}

When $0< x< 1$, recall that one has the following expression 
as indicated by Ingham:
$$
{\sum_{n \leq 1/x}}'~~ \frac{\Lambda(n)}{n} 
= 
-\log x - \gamma ~+~ \sum_{\rho}\frac{x^\rho}{\rho}
~+~ \frac{1}{2}\log \frac{1+x}{1-x} -x, 
$$
where $\gamma$ denotes the Euler's constant.
As noted earlier, this is deduced  by considering 
the following integral
$$
\frac{1}{2\pi i}\int_{3-i \infty}^{3+ i \infty} \frac{x^{1-s}}{1-s} 
~\frac{\zeta'}{\zeta}(s) ~ds.
$$  

A similar argument can be applied to an arbitrary 
element $F$ in the Selberg
class.   Let us write
$$
-\frac{F'}{F}(s) = \frac{m_F}{s-1} - \gamma_F + O(s-1). 
$$
Then, for $ x \in (0,1)$ such that $1/x$ is not a prime power, 
we have
$$
\sum_{n\leq 1/x} \frac{\Lambda_F(n)}{n} 
= 
-m_F\log x - \gamma_F 
+ \sum_{\rho}\frac{x^{\rho}}{\rho} 
+ \sum_{j=1}^r \left\{ \lambda_j x^{1+\mu_j/\lambda_j}
f_{\lambda_j + \mu_j} (x^{1/\lambda_j}) 
+ \frac{\lambda_j x^{1+\mu_j/\lambda_j}}{\lambda_j + \mu_j}\right\},
$$
where $\rho$ runs over the non-trivial  zeros of $\overline{F}(s)$.
Finally, when $x \in (0,1)$ and  $\alpha \ne 0$ and also not equal to 
poles and zeros of $F$, we have 
\begin{eqnarray*}
T(x, F, \alpha)  
 &=&
 \sum_{\rho}\frac{x^\rho}{\rho - \alpha} + \frac{m_F}{\alpha}   
 +    \frac{m_Fx}{1- \alpha}-
 x^{\alpha} \frac{{F}'}{F}(1- \alpha)   \\
 &&
  + \sum_{j=1}^r \lambda_j \frac{x^{1+(\mu_j/\lambda_j)}}{\mu_j+
\lambda_j- \alpha \lambda_j} 
+ \sum_{j=1}^r \lambda_j  x^{1+ \mu_j/\lambda_j} 
f_{\mu_j + \lambda_j- \alpha\lambda_j}(x^{1/\lambda_j}), 
\end{eqnarray*}
where $\rho$ runs over the non-trivial  zeros of $\overline{F}(s)$.
Here
$$
 T(x, F, \alpha) := 
\begin{cases}
x^{\alpha} {\sum_{n \leq 1/x}}~~ \frac{\Lambda_F(n)}{n^{1-\alpha}} 
& \text{if $1/x$ is not a prime power;}\\ \\
x^{\alpha} {\sum_{n < 1/x}}~~ \frac{\Lambda_F(n)}{n^{1-\alpha}}  + 
\frac{x}{2}\Lambda_F(1/x) & \text{otherwise.}
\end{cases}
$$
Hence we have 
\begin{eqnarray*}
\sum_{\rho}\frac{x^\rho}{\rho - \alpha} -
 x^{\alpha} \frac{{F}'}{F}(1- \alpha) 
 &=&
 T(x, F, \alpha)    - \frac{m_F}{\alpha}   
 -    \frac{m_Fx}{1- \alpha}
- \sum_{j=1}^r \frac{\lambda_j x^{1+(\mu_j/\lambda_j)}}{\mu_j
+ \lambda_j- \alpha \lambda_j}  \\
&&
-~~~ \sum_{j=1}^r \lambda_j  x^{1+ \mu_j/\lambda_j} f_{\mu_j 
+ \lambda_j- \alpha\lambda_j}(x^{1/\lambda_j}).
\end{eqnarray*}

Recalling that two functions $F,G\in \S$ are of the same {\it Hodge type}
if they admit a functional equation with the same  $\lambda_j$ and 
the $\mu_j$, we have the following theorem.

\begin{thm}  
Let $A(t), B(t)$ be as in \thmref{ext}.
Consider the set of elements $F\in \S$ with a fixed Hodge type
such that 
$$
h(x) := 
\sum_{\rho} \frac{A(\rho)}{B(\rho)}x^\rho
 - \sum_{i=1}^d \frac{A(\alpha_i)}{B'(\alpha_i)} \frac{{F}'}{F}
 (1-\alpha_i) x^{\alpha_i}
$$
is algebraic for algebraic $x <1$. Further assume that 
$\beta_i=\frac{A(\alpha_i)}{B'(\alpha_i)}$
are real with same sign for $1\le i \le d$.  
If there are two elements $F,G$ in this set, then for any 
prime $p$ with $1/\sqrt{x} < p \le 1/x$, we have
$$
b_F(p) =  b_G(p). 
$$
\end{thm}
\begin{proof}
The proof follows arguing along the line of proof 
of \thmref{hodge}.
\end{proof}

Finally, when $F$ is in the Arithmetic Selberg class,
we have the following theorem whose proof 
is analogous to that of \thmref{arith}.
 \begin{thm}\label{ext}
 Let $F\in \A$. Let $ A(t)$ and $B(t)$ be as before and let  
 $\alpha_{1} ..., \alpha_{d}$ be the  roots of $B(t)$
 which are all rational, non-zero, simple and not equal 
 to the zeros and poles of $F$. Then
for an algebraic $x \in (0,1)$, the number
$$
h(x) := 
\sum_{\rho} \frac{A(\rho)}{B(\rho)}x^\rho
 - \sum_{i=1}^d \frac{A(\alpha_i)}{B'(\alpha_i)} \frac{{F}'}{F}
 (1-\alpha_i)
 x^{\alpha_i}
$$
is an extended Baker period. Here $\rho$ 
runs over the non-trivial  zeros of $\overline{F}(s)$.
Further suppose that $\mu_j =0 ~~~~~\forall j$. Then
\begin{itemize}
\item
 if $m_F \sum_{i=1}^{d}
\frac{\beta_{i}}{1-\alpha_{i}} \neq 0$, then
$h(x)$ has at most one algebraic zero 
in~$(0,1)$;
\item
if $m_F \sum_{i=1}^{d} 
\frac{\beta_{i}}{1-\alpha_{i}} = 0$, then the
set
$$
\{ h(x) ~|~ x \in (0,1)\cap \overline{\Q} \}
$$
has at most one algebraic number.
\end{itemize}
\end{thm}
\section{Concluding remarks}
It is yet unclear what role (if any) transcendental number theory
plays in our journey towards the grand Riemann hypothesis.  The generalized
Li criterion as well as many of the theorems of this paper suggest that
there may be a link.  If so, this paper represents a humble
beginning towards our lofty goal.

\end{document}